\documentclass[reqno,12pt]{amsart}
\usepackage{amsmath}
\usepackage{latexsym,amsmath,MnSymbol}
\usepackage{amsthm}
\usepackage{graphics}

\newtheorem{theorem}{Theorem}     

\newtheorem{lemma}[theorem]{Lemma}
\newtheorem{fact}[theorem]{Fact}

\newtheorem{conjecture}[theorem]{Conjecture}

\newtheorem{case}{Case}
\newtheorem{claim}{Claim}[case]

\def\D{\Delta}

\def\eps{\varepsilon}

\def\Z{\mathbb{Z}}
\def\pr{\mathbb{P}}
\def\ex{\mathbb{E}}

\newcounter{rot}

\title[Total Edge Irregularity Strength]
{Total Edge Irregularity Strength of Large Graphs}
\author[F.Pfender]{Florian Pfender}

\begin{document}

\begin{abstract}
Let $m:=|E(G)|$ sufficiently large and $s:=\left\lceil\frac{m-1}{3}\right\rceil$.
We show that unless the maximum degree $\Delta > 2s$, there is a weighting $\hat{w}:E\cup V\to \{ 0,1,\ldots,s\}$ 
so that $\hat{w}(uv)+\hat{w}(u)+\hat{w}(v)\ne \hat{w}(u'v')+\hat{w}(u')+\hat{w}(v')$ whenever $uv\ne u'v'$ 
(such a weighting is called {\em total edge irregular}).
This validates a conjecture by Ivan\v{c}o and Jendrol' for large graphs, extending a result by Brandt, Mi\v{s}kuf 
and Rautenbach.
\end{abstract}

\maketitle

\section{Introduction}
Let $G=(V,E)$ be a graph.
In~\cite{BJMR}, Ba\u{c}a, Jendrol', Miller and Ryan define the notion of an {\em edge irregular total $s$-weighting} as a weighting  \[\hat{w}:E\cup V\to \{ 1,2,\ldots,s\}\]
so that 
\[\hat{w}(uv)+\hat{w}(u)+\hat{w}(v)\ne \hat{w}(u'v')+\hat{w}(u')+\hat{w}(v')\] whenever $uv\ne u'v'$ are two different edges of $G$. They also define the {\em total edge irregularity strength} as the minimum $s$ for which there exists such a weighting, denoted by $tes(G)$. If we denote by $\D$ the maximum degree of $G$ and by $m$ the number of edges they note that 
\[
 tes(G)\ge \max\left\{ \frac{m+2}{3},\frac{\D+1}{2}\right\}.
\]
After some more study of $tes(G)$, Ivan\u{c}o and Jendrol' conjecture in~\cite{IJ} that this natural lower bound is sharp for all graphs other than the complete graph on $5$ vertices (which has $tes(K^5)=5$), i.e.,
\begin{conjecture}[Ivan\u{c}o and Jendrol'~\cite{IJ}]\label{con}
 For every graph $G$ with $|E(G)|=m$ and maximum degree $\D$ which is different from $K^5$,
\[
 tes(G)= \max\left\{ \left\lceil\frac{m+2}{3}\right\rceil,\left\lceil\frac{\D+1}{2}\right\rceil\right\}.
\]
\end{conjecture}
Conjecture~\ref{con} has been verified for trees in~\cite{IJ}, for complete graphs and complete bipartite graphs by Jendrol', Mi\u{s}kuf and Sot\'{a}k in~\cite{JMS}, and for graphs with a bound on $\D$ by Brandt, Mi\u{s}kuf and Rautenbach in~\cite{BMR} and~\cite{BMR2}:
\begin{theorem}[Brandt et.al~\cite{BMR} and~\cite{BMR2}]\label{brandt}
 For every graph $G$ with $|E(G)|=m$ and maximum degree $\D$, where $\left\lceil\frac{\D+1}{2}\right\rceil\ge \frac{m+2}{3}$ or $\D\le \frac{m}{111000}$,
\[
 tes(G)= \max\left\{ \left\lceil\frac{m+2}{3}\right\rceil,\left\lceil\frac{\D+1}{2}\right\rceil\right\}.
\]
\end{theorem}

%

In this paper, we show the conjecture for all sufficiently large graphs.

\begin{theorem}\label{main}
 Let $G$ be a graph with $m:=|E(G)|\ge 7\times 10^{10}$ and maximum degree $\D$.
Then 
\[
 tes(G)= \max\left\{ \left\lceil\frac{m+2}{3}\right\rceil,\left\lceil\frac{\D+1}{2}\right\rceil\right\}.
\]
\end{theorem}
The proof of this Theorem will be presented in Section~\ref{mainproof}. With a similar proof, presented in Section~\ref{maxdegproof}, we can improve on Theorem~\ref{brandt} as follows.
\begin{theorem}\label{maxdeg}
 Let $G$ be a graph with $m:=|E(G)|$, and $\Delta(G)\le \frac{m}{4350}$. 
Then 
$
 tes(G)=  \left\lceil\frac{m+2}{3}\right\rceil.
$
\end{theorem}

For notation not defined here, we refer the reader to Diestel's book~\cite{D}. In particular, if $X$ and $Y$ are subsets of the vertex set of a graph $G$ and if $E'\subseteq E$ is a subset of its edges, we write $G[X]$ for the induced subgraph of $G$ on $X$, and we write short $E'(X)$ for the edge set $E'\cap E(G[X])$, and $E'(X,Y)$ for all edges in $E'$ from $X$ to $Y$. 

\section{Preliminary Results}

By Theorem~\ref{brandt}, we only have to consider the case $\left\lceil\frac{\D+1}{2}\right\rceil< \frac{m+2}{3}$.
Without loss of generality we may assume in the following that $m-1$ is 
divisible by $3$, as otherwise we may just add one or two edges (and possibly vertices) and consider the 
larger graph, only increasing the difficulty of the assignment. 

Let $s:=\frac{m-1}{3}$ and $w:V\to \{ 0,1,\ldots,s\}$ be a vertex weighting. For $e=xy$ we set $w(e):=w(x)+w(y)$. 
We call $w$ {\em well guarded} if for all $0\le i\le 2s$, 
\[
i+1\le |\{ e\in E\mid w(e)\le i\} |\le i+s+1.
\] 
The following fact is immediate as well 
guarded weightings are easily extended to total edge irregular weightings and vice versa:
\begin{fact}
 A graph $G$ has a total edge irregular weighting $\hat{w}:E\cup V\to \{ 1,2,\ldots,s+1\}$ if and only if
$G$ has a well guarded weighting $w:V\to \{ 0,1,\ldots,s\}$.
\end{fact}
Thus, we can restrict ourselves to vertex weightings in our quest for total edge irregular 
weightings. We will call an edge set $E'\subseteq E$ a {\em guarding set}, if for all $0\le i\le 2s$,
\[
i+1\le |\{ e\in E'\mid w(e)\le i\} |
\le i+s+1-|E\setminus E'|.
\]
Clearly,  $w$ is well guarded if and only if a guarding set exists.
The next lemma describes a set up where we can find a guarding set deterministically. In the proof of Theorem~\ref{main}, we will encounter this set up several times.
\begin{lemma}
 \label{easy}
Let $V(G)=A_1\cupdot A_2\cupdot C$ be a partition of the vertices of a graph $G$ with $3s+1=m=|E(G)|$, let $E'=E(G)\setminus E(C)$, and let $\Delta_i:=\max_{v\in C}{|E(A_i,v)|}$ for $i\in \{ 1,2\}$. 
If 
\begin{enumerate}
\item $|E(A_1)|\le |E'|-2s-\Delta_1$,\label{a}
\item $|E(A_2)|\le |E'|-2s-\Delta_2$,\label{b}
\item $|E(A_1,V)|\le |E'|-s+1-\Delta_2$,\label{c}
\item $|E(A_2,V)|\le |E'|-s+1-\Delta_1$,\label{d}
\item $\Delta_2+\tfrac{s}{|E(A_1,C)|-\Delta_1}\Delta_1\le s-|E\setminus E'|$.\label{e}
\end{enumerate}
then there exists a weighting such that $E'$ is a guarding set.
\end{lemma}
\begin{proof}
Let $C=\{ x_1,x_2,\ldots x_{|C|}\}$, where the exact order will be determined later. Let $C'=C- x_{|C|}$.
Let 
\[
 w(v)=\begin{cases}
       0, &\mbox{ for }v\in A_1,\\
       s, &\mbox{ for }v\in A_2,\\
\left\lceil \tfrac{s\cdot |E(A_1,\{ x_1,\ldots ,x_{i-1}\})|}{|E(A_1,C')|}\right\rceil,
&\mbox{ for }v=x_i\in C.
      \end{cases}
\]
Then for $0\le i< s$, 
we have
\begin{multline*}
 |E(A_1)|+\left\lfloor i\cdot\tfrac{ |E(A_1,C')|}{s}\right\rfloor +1\\
\le |\{ e\in E'\mid w(e)\le i\} | \\
\le |E(A_1)|+i\cdot \tfrac{|E(A_1,C')|}{s}+\Delta_1,
\end{multline*}
and therefore
\[
 i+1\le_{\eqref{d}} |\{ e\in E'\mid w(e)\le i\} | \le_{\eqref{a},\eqref{c}} i+1+s-|E\setminus E'|,
\]
regardless of the order of the vertices in $C$.

For $s\le i\le 2s$, we can now find a suitable ordering of $C$ greedily to show the lemma.
Pick $x_{|C|}$ first, so that $|E(A_2,x_{|C|})|$ is minimized under the condition that
\[
 |E(A_2)|+|E(A_2,x_{|C|})|\ge \Delta_2.
\]
Now choose the other $x_j$, starting with an arbitrary $x_1$, such that for every $j$,
\begin{multline}\label{ie1}
 s+w(x_j)
\le |E(A_1,V)|+|E(A_2,\{ x_1,\ldots,x_{j-1}\})|\\
\le 2s+1-|E\setminus E'|+w(x_j)-\Delta_2.
\end{multline}
This is always possible, as this inequality is true for $x_1$ (by~\eqref{a} and~\eqref{c}) and $x_{|C|}$ (by~\eqref{b}), and at no point in the process there can be remaining $x,x'\in C$ such that setting $x_j=x$ violates the lower inequality, and setting $x_j=x'$ violates the upper inequality by~\eqref{e}.
As
\[
|\{ e\in E'\mid w(e)\le i\} |=|E(A_1,V)|+|E(A_2,\{ x_1,\ldots,x_j\})|,
\]
for $j\le |C|$ maximized such that $w(x_j)\le i-s$, this shows that $E'$ is a guarding set.
\end{proof}

\section{Proof of Theorem~\ref{main}}\label{mainproof}


Let $\eps=2.7\times 10^{-5}$, 
and define the set of large degree vertices
\[
 B:=\{ v\in V\mid d(v)>\eps m\}.
\]
Then $m\ge |B|\eps m-\frac{|B|^2}{2}$, and therefore $|E(B)|<\frac{|B|^2}{2}<0.01m$.

Let $V'=V\setminus B$ and $m'=|E\setminus E(B)|>0.99m$.
Further, we partition $B$ into $B_0$ and $B_S$ as follows:
Order the vertices in $B$ by degree from large to small, and assign them in order to the set with fewer edges to $V'$. 
Let $e_0:= \frac{1}{m'}|E(B_0,V')|$ and
$e_S:= \frac{1}{m'}|E(B_S,V')|$.
Observe the following fact.
\begin{fact}\label{dist}
 If 
$v\in B_0$, then $e_0-e_S\le \frac{1}{m'}d(v).$ 
\end{fact}
We will divide the proof into four cases. For the first three, we assume that $e_0\ge e_S$.

\begin{case}$e_0\ge 0.52$ and $|B_0|=1$.
\label{case1}\end{case}
Let $v_1$ be the vertex in $B_0$, then $d(v_1)=\Delta :=\Delta(G)>0.51m$, and let $H=G[V\setminus v_1]$. 
Note that in this case, we may assume that $V=\{v_1\}\cup N(v_1)$ (so $|V(H)|=\Delta$ and $|E(H)|=m-\Delta $). Otherwise, as $H$ does not have enough edges to be connected, a vertex $u\in V(H)\setminus N(v_1)$ has distance at least $3$ in $G$ to a vertex $v$ in another component of $H$. We can identify these two vertices and proceed with the smaller graph $G'$, where $|E(G')|=|E(G)|$ and $tes(G')\ge tes(G)$.

\begin{claim}\label{sub}
There exists $X'\subseteq V'$ with 
\[ |X'|= \lfloor\tfrac{2}{3}(2|V(H)|-|E(H)|)\rfloor \mbox{ and }|E(X')|\le \tfrac{1}{2}|X'|.\]
\end{claim}
Let $X'\subseteq V'$ with $|X'|=\lfloor\frac{2}{3}(2|V(H)|-|E(H)|)\rfloor$ and $|E(X')|$ minimal, and let $Y'=V'\setminus X'$. If $|E(X')|>\frac{1}{2}|X'|$, then $|E(y,X')|\ge 2$ for all $y\in Y'$, as otherwise we could reduce $|E(X')|$ by a vertex switch. Thus,
\begin{multline*}
|E(H)|\ge |E(X')|+|E(X',Y')|+|E(B,V)|-\Delta\\
>\frac{1}{2}|X'|+2|Y'|+\tfrac{1}{2}(|B|-1)\eps m\\
\ge 2|V(H)|-\frac{3}{2}|X'|\ge |E(H)|,
\end{multline*}
a contradiction showing the claim.
\begin{claim}\label{c2}
There exists  $X\subseteq V'$ with $|X|\ge s+1$ and $|E(X)|\le 2s-\Delta +1$.
\end{claim}
Use Claim~\ref{sub} to find a vertex set $X'\subseteq V'$.
 Successively delete vertices of maximum degree in $X'$ until we have a vertex set $X\subseteq X'$ with $|X|=s+1$. Then either $|E(X)|=0$ or 
\begin{multline*}
|E(X)|\le |E(X')|-(|X'|-|X|)\\
\le |X|-\lceil\tfrac{1}{2}|X'|\rceil
\le s+1-\tfrac{1}{3}(2|V(H)|-|E(H)|)+\tfrac{1}{2}\\
\le s+1-(\Delta-s-\tfrac{1}{3})+\tfrac{1}{2}
=2s-\Delta+\tfrac{11}{6},
\end{multline*}
showing the claim.

Now choose $X$ according to Claim~\ref{c2}, maximizing $|X|$,
and let $Y:=V(H)\setminus X$. 
%
We want to use Lemma~\ref{easy} to show that $E'=E\setminus E(X)$ is a guarding set: Let $A_1=\{ v_1\}$, $A_2=Y$ and $C=X$. Then $\Delta_1=1$ and $\Delta_2\le \eps m$. Conditions \eqref{a}, \eqref{d} and \eqref{e} are easily verified.

If \eqref{c} fails, say $|E(A_2)|+|E(A_2,C)|=s-1+\Delta_2-\gamma$, note that $X$ contains at least $|X|-s+\gamma$ vertices with no neighbors in $Y$. If \eqref{b} holds, we can use these vertices first in the proof of Lemma~\ref{easy} until \eqref{ie1} is satisfied, and see that $E'$ is a guarding set.

Finally, assume that \eqref{b} fails, i.e., 
\[|E(Y)|>|E'|-2s-\Delta_2\ge |E'|-2s-\eps m.\]
As every vertex in $Y\setminus B$ has at least one neighbor in $X$ by the maximality of $|X|$, 
we have 
\[
 |Y|\le |E'|-\Delta-|E(Y)|+|B|<2s-\Delta+\eps m+\tfrac{2}{\eps}\le 2s-\Delta+3\eps m.
\]

Let $Y_1:=\{ y\in Y: d(y)\ge 0.01s\}$. Then $0.48m\ge 0.01s|Y_1|-0.5|Y_1|^2$, so $|Y_1|\le 160$ and $|E(Y_1)|<0.02s$ as $s> 10^6$. Let $Y_2=Y\setminus Y_1$. Then more than $s-|E(X,Y)|-|E(Y_1)|>0.47s>2s-\Delta$ edges in $E(X\cup Y)$ are incident to $Y_2$, so we can greedily find some $Y_3\subseteq Y_2$ with 
\[
2s-\Delta\le |E(Y_3,X\cup Y)|< 2s-\Delta+0.01s.
\]
Let $a:=|E(Y_3)|$, and $b:=|E(Y_3,X\cup Y)|-a$. 

Let $X=\{ x_0,\ldots ,x_{|X|-1}\}$ with $|E(Y,x_i)|\le |E(Y,x_j)|$ for $i\le j$ and let 
$w(v_1)=0$, $w(v)=s-c:=\min\{ s-b,\lceil \Delta/2\rceil\}$ for $v\in Y_3$, $w(v)=s$ for $v\in Y\setminus Y_3$,
and 
$w(x_i)=\min\{ s,i\}$. 
We claim that $E'$ is a guarding set for $w$, settling Case~\ref{case1}.


For $0\le i\le s-1$, we have
\begin{multline*}
 i+1\le |\{ e\in E'\mid w(e)\le i\}|\le i+1+|Y_3|\\
<i+1+2s-\Delta<i-s+\Delta=i+s+1-|E\setminus E'|.
\end{multline*}
For $i=2s$, we have
\[
 2s+1< \{ e\in E'\mid w(e)\le 2s\}|=|E'|= 3s+1-|E\setminus E'|.
\]
For $s\le i\le 2s-1$, consider first the lower bound. We have
\[
 |\{ e\in E'\mid w(e)\le i\}|\ge \left.
\begin{cases}
 \Delta,&\mbox{ for }s\le i<2s-2c,\\
 \Delta+a,&\mbox{ for }2s-2c\le i<2s-c,\\
 \Delta+a+b,&\mbox{ for }2s-c\le i<2s,
\end{cases}
\right\} 
\ge s+i+1.
\]
For the sake of analysis of the upper bound, define another weighting $w'$, where $w'(v)=s$ for $v\in Y_3$ and $w'=w$ on all other vertices. Then for $s\le i<2s$,
\begin{multline*}
 \Delta\le |\{ e\in E'\mid w'(e)\le i\}|\le \max\{ \Delta, i+s-|X|-1\}\\
\le \max\{ \Delta, i+3s-2\Delta-2\}<
\max\{ \Delta, i-0.06s\}
\end{multline*}
since $|\{ e\in E'\mid w'(e)\le i\}|$ is maximized if $\max_{x\in X} |E(x,Y)|=1$.

Any edge $e\in E(X\cup Y)$ with $w'(e)<2s$ has weight $w(e)\ge w'(e)-c$. Therefore, we have
\begin{multline*}
 |\{ e\in E'\mid w(e)\le i\}|\\
 \le\left.
\begin{cases}
 |\{ e\in E'\mid w'(e)\le i+c\}|,&\mbox{ for }s\le i<2s-2c,\\
 |\{ e\in E'\mid w'(e)\le i+c\}|+a,&\mbox{ for }2s-2c\le i<2s-c,\\
 |\{ e\in E'\mid w'(e)\le 2s-1\}|+a+b,&\mbox{ for }2s-c\le i<2s,
\end{cases}
\right\} \\ 
\le\left.
\begin{cases}
 \max\{ \Delta, i-0.06s\},&\mbox{ for }s\le i<2s-2c,\\
 \max\{ \Delta, i-0.06s\}+a,&\mbox{ for }2s-2c\le i<2s-c,\\
 \max\{ \Delta, i-0.06s\}+a+b,&\mbox{ for }2s-c\le i<2s,
\end{cases}
\right\} \\
=\left.
\begin{cases}
 \Delta,&\mbox{ for }s\le i<2s-2c,\\
 \max\{ \Delta, i-0.06s\}+a,&\mbox{ for }2s-2c\le i<2s-c,\\
 \max\{ \Delta, i-0.06s\}+a+b,&\mbox{ for }2s-c\le i<2s,
\end{cases}
\right\} \\
\le \Delta+i-s=s+i+1-|E\setminus E'|.
\end{multline*}
To see the last inequality, note that it is enough to check it for $i\in \{ s,2s-2c,2s-c\}$. For $i=1$, the inequality is trivially true. For $i=2s-2c$, we have
\begin{multline*}
 \Delta +a<2.01s-b<\min\{ 2\Delta-1.01s-b,2\Delta-s\}\\
\le\min\{ \Delta+s-2b,2\Delta-s\}\le \Delta+i-s.
\end{multline*}
For $i=2s-c$, we have
\begin{multline*}
 \Delta +a+b<2.01s<\min\{ 2\Delta-1.01s,\lceil\tfrac{3}{2}\Delta\rceil\}\\
\le\min\{ \Delta+s-b,\lceil\tfrac{3}{2}\Delta\rceil\}= \Delta+i-s,
\end{multline*}
and 
\begin{multline*}
 i-0.06s +a+b=1.94s-c+a+b<
3.95s-\Delta-c
\le\Delta+s-c= \Delta+i-s.
\end{multline*}
This shows that $E'$ is a guarding set, establishing Case~\ref{case1}.

\begin{case}
$e_0\ge 0.52$ and $|B_0|=2$.
\end{case}


Let $B_0=\{ v_2,v_3\}$ and $v_1\in B_S$ with $d(v_1)=\Delta(G)$.
Let $H=G[V\setminus \{v_1,v_2,v_3\}]$. Let $d_i:=|E(v_i,V(H))|$ for $1\le i\le 3$, and we may assume that $d_1\ge d_2\ge d_3$. Note that $d_2+d_3>0.51m$.
Then $|H|\ge d_1$, and 
$|E(H)|\le m-d_1-d_2-d_3< 0.24m$.

\begin{claim}\label{c3}
 There is a set $X'\subseteq V'$ such that
 \[|E(X',\{ v_2,v_3\} )|\ge \min\{d_2+d_3-|B|, \tfrac{4}{3}(d_1+2d_2+2d_3-m)-2\},\] and
$|E(X')|\le 0.25|E(X',\{ v_2,v_3\} )|$.
\end{claim}

Let $X'\subseteq V'\cap (N(v_2),N(v_3))$ with 
\begin{multline*}
  \min\{d_2+d_3-|B|, \tfrac{4}{3}(d_1+2d_2+2d_3-m-2)\}\le |E(X',\{ v_2,v_3\} )|\\
< \tfrac{4}{3}(d(v_1)+2d(v_2)+2d(v_3)-m),
\end{multline*}
such that $|E(X')|$ is minimal.
Let $Y'=V'\setminus X'$. 
 If $|E(X')|>0.25|E(X',\{ v_2,v_3\} )|$, then $|E(y,X')|\ge |E(y,\{ v_2,v_3\} )|$ for all $y\in Y'$, as otherwise we could reduce $|E(X')|$ by a vertex switch. 
Thus,
\begin{multline*}
m-d_1-d_2-d_3\ge |E(H)|\ge |E(X')|+|E(X',Y')|+\tfrac{1}{2}(|B|-3)\eps m\\
>0.25|E(X',\{ v_2,v_3\} )|+|E(Y',\{ v_2,v_3\} )|+\tfrac{1}{2}(|B|-3)\eps m\\
\ge d_2+d_3-0.75|E(X',\{ v_2,v_3\} )|\\
> m-d_1-d_2-d_3
\end{multline*}
a contradiction showing the claim.

\begin{claim}\label{c4}
 There is a set $X\subseteq V'$ such that
 $|E(X,\{ v_2,v_3\} )|\ge s+2$, and
\[|E(X)|\le \max\{ 0.5s-\tfrac{1}{4}(d_2+d_3-|B|)+1.5,1.5s-\tfrac{1}{3}(d_1+2d_2+2d_3)+2.5\}.\]
\end{claim}
Start with a set $X'$ from Claim~\ref{c3}, and succesively delete vertices maximizing $\frac{|E(v,X)|}{|E(v,\{v_2,v_3\})|}$ until $|E(X,\{v_2,v_3\})|\le s+3$. Then either $|E(X)|=0$ or
\begin{multline*}
 |E(X)|\le |E(X')|-0.5(|E(X',\{v_2,v_3\})|-s-3)\\
\le -0.25|E(X',\{ v_2,v_3\} )|+0.5(s+3)\\
\le  0.5s+1.5-\min\{ \tfrac{1}{4}(d_2+d_3-|B|),\tfrac{1}{3}(d_1+2d_2+2d_3-m-2)\}\\
\le \max\{ 0.5s-\tfrac{1}{4}(d_2+d_3-|B|)+1.5,1.5s-\tfrac{1}{3}(d_1+2d_2+2d_3)+2.5\},
\end{multline*}
showing the claim.

We want to apply Lemma~\ref{easy} to this situation with $A_1=\{v_2,v_3\}$, $C=X$ for a maximal $X$, and $A_2=V\setminus(A_1\cup C)$. 
Conditions \eqref{a}, \eqref{d} and \eqref{e} are clearly satisfied. 
For condition~\eqref{b} note that by the maximality of $|X|$, every vertex in $A_2\setminus B$ has a neighbor in $X$, so in particular, $ |E(A_2,C)|\ge d_1-|B|\ge d_1-2\eps m$, and so
\[
 |E(A_1,V)|+|E(A_2,C)|\ge d_1+d_2+d_3-2-2\eps m>0.7m>2s+\Delta_2.
\]
If condition~\eqref{c} holds, we are done by Lemma~\ref{easy}. Finally, if condition~\eqref{c} fails, we have
\begin{align*}
 d_2+d_3&\ge |E(A_1,V)|\\
&>|E'|-s+1-\Delta_2\\
&=2s+2-|E(X)|-\Delta_2\\
&\ge \min\{ 1.49s+ \tfrac{1}{4}(d_2+d_3),0.49s+\tfrac{5}{6}(d_2+d_3)\},
\end{align*}
and therefore
\[
 d_2+d_3>1.98s,~d_1+d_2+d_3>2.97s,\mbox{ and }d_3>0.96s.
\]
Now it is easy to construct a weighting with guarding set $E(\{v_1,v_2,v_3\},V)$, starting with $w(v_1)=0$, $w(v_2)=s$ and $w(v_3)=\lceil 0.5s\rceil$.

%

\begin{case}
$|B_0|\ge 3$ and $e_0+e_S\ge 0.86$.
\end{case}
Since $|B_0|\ge 3$, we have $e_S\ge \tfrac{2}{3}e_0$. Thus, this case covers all remaining situations with $e_0\ge 0.52$. Let $A_1=B_0$, $A_2=B_S$, and $C=V\setminus (A_1\cup A_2)$. 
Then $|E'|\ge (e_0+e_S)m'>2.52s$ and $\Delta_i<\eps m$.
All conditions of Lemma~\ref{easy} apply but possibly~\eqref{c}.
If~\eqref{c} fails, we have
\[
 s-1+\Delta_2>|E(B_S)|+|E(B_S,C)|
\ge e_sm'> 0.344m'>s+\Delta_2,
\]
a contradiction finishing the case.

For the last case, we will drop the assumption of $e_0\ge e_S$ to be able to use symmetry in a different place.

\begin{case}$\max\{ e_0,e_S\} \le 0.52$.\end{case}

Let $w(v)=0$ for $v\in B_0$, $w(v)=s$ for $v\in B_S$, and determine $w(v)$ for all other 
vertices independently at random with
\[
 P_i:=\pr \left( w(v)=s\frac{i}{20}\right) =
\begin{cases}
 (19-30e_0)\beta^{-1} & i=0\\
 \beta^{-1} & 0<i<20\\
 (19-30e_S)\beta^{-1} & i=20\\
 0 & \mbox{otherwise}
\end{cases}
,
\]
where $\beta=57-30(e_0+e_S)$ and $i\in \Z$.

The set $E'=E\setminus E(B)$ is guarding for the resulting weighting, if for $0\le i\le 39$,
\[
 \tfrac{i+1}{20}s\le |\{ e\in E'\mid w(e)\le  s\tfrac{i}{20}\}|\le \tfrac{i}{20}s+s+1-|E(B)|.
\]
To show that $E'$ has a positive probability of being a guarding set, 
we will use Azuma's inequality. For this, let us first consider the 
expected number of edges of the particular weights 
\[
 X_i:=|\{ e\in E'\mid w(e)\le  s\tfrac{i}{20}\}|
\]
and find values $\delta_i,\hat{\delta}_i\in (0,0.01)$ 
such that for $0\le i \le 39$, 
\begin{equation}\label{bound1}
 \ex(X_i)\ge
\begin{cases}
 \tfrac{i+1}{20}s+\delta_im', &\mbox{ if }e_0\ge e_S,\\
 \tfrac{i+1}{20}s+\hat{\delta}_im', &\mbox{ if }e_0\le e_S,\\
 \tfrac{i+1}{20}s+0.2m', &\mbox{ if }20\le i\le 39,
\end{cases}
\end{equation}
and
\begin{equation}\label{bound2}
 \ex(X_i)\le
\begin{cases}
 \tfrac{i}{20}s+s+1-|E(B)|-\delta_im', &\mbox{ if }e_0\ge e_S,\\
 \tfrac{i}{20}s+s+1-|E(B)|-\hat{\delta}_im', &\mbox{ if }e_0\le e_S,\\
 \tfrac{i}{20}s+s+1-|E(B)|-0.2m', &\mbox{ if }0\le i\le 19.
\end{cases}
\end{equation}

%

By symmetry (change the sets $B_0$ and $B_S$), we can set $\delta_i=\hat{\delta}_{39-i}$, so we
only have to treat the cases $0\le i\le 19$.
For every edge $uv\in E'$, we get 
\[
 \pr (w(uv)=s\tfrac{i}{20})=\sum P_kP_{i-k}=2P_0P_i+(i-1)P_i^2.
\]
For $uv\in E'(B_0,V)$, we get
$\pr (w(uv)=s\frac{i}{20})=P_i$, and for $uv\in E'(B_S,V)$, we get
$\pr (w(uv)=s\frac{i}{20})=P_{i-20}=0$. Thus, 
\begin{multline*}
 \tfrac{1}{m'}\ex(X_i)=
\sum_{j=0}^i (e_0P_j+(1-e_0-e_S)(2P_0P_j+(j-1)P_j^2))\\
= e_0(P_0+iP_1)+(1-e_0-e_S)(P_0^2+2iP_0P_1+\tfrac{i^2-i}{2}P_1^2).
%
%
\end{multline*}

Fixing $i$ and taking the partial derivatives, your favorite computer algebra program tells you that
$\frac{d}{de_0}\ex(X_i)=0$ if and only if
\[
e_0=f_1(e_S):=-\frac{p_0+p_1e_S+p_2e_S^2}{
 p_3+p_4e_S},
\]
and $\frac{d}{de_S}\ex(X_i)=0$ if and only if
\[
e_0=f_2(e_S):=
\frac{p_5+p_6e_S-\sqrt{p_7+p_8e_S+p_9e_S^2}}{p_{10}},
\]
where the $p_j$ are polynomials in $i$ (see Appendix~\ref{polys}).
\begin{figure}[ht]
 \scalebox{0.55}{\includegraphics{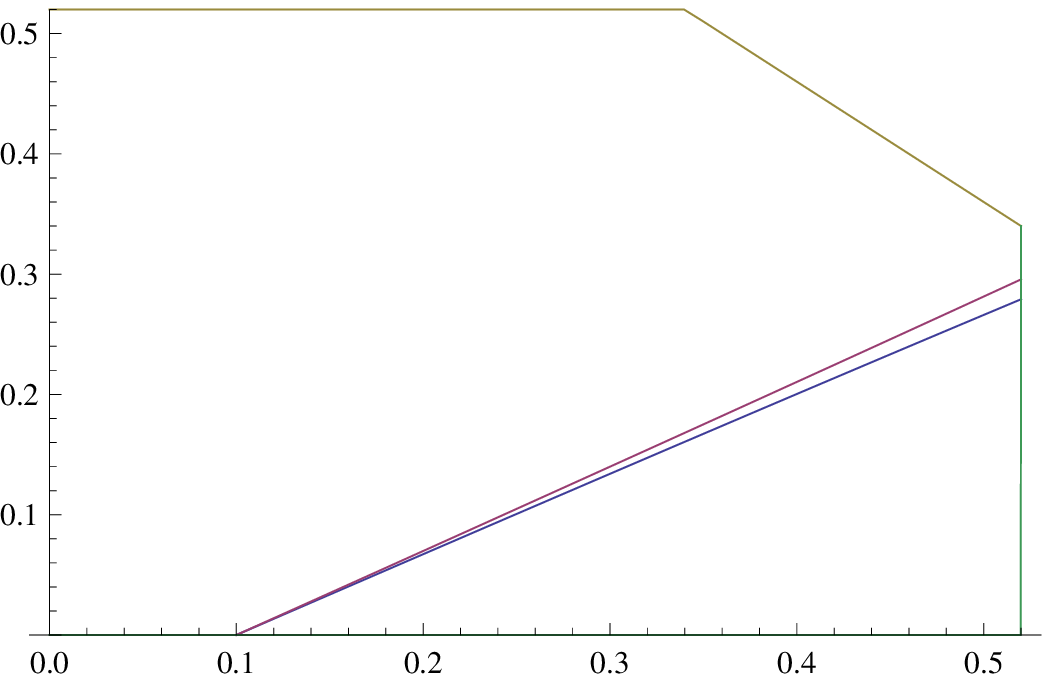}}\hspace{1cm}\scalebox{0.55}{\includegraphics{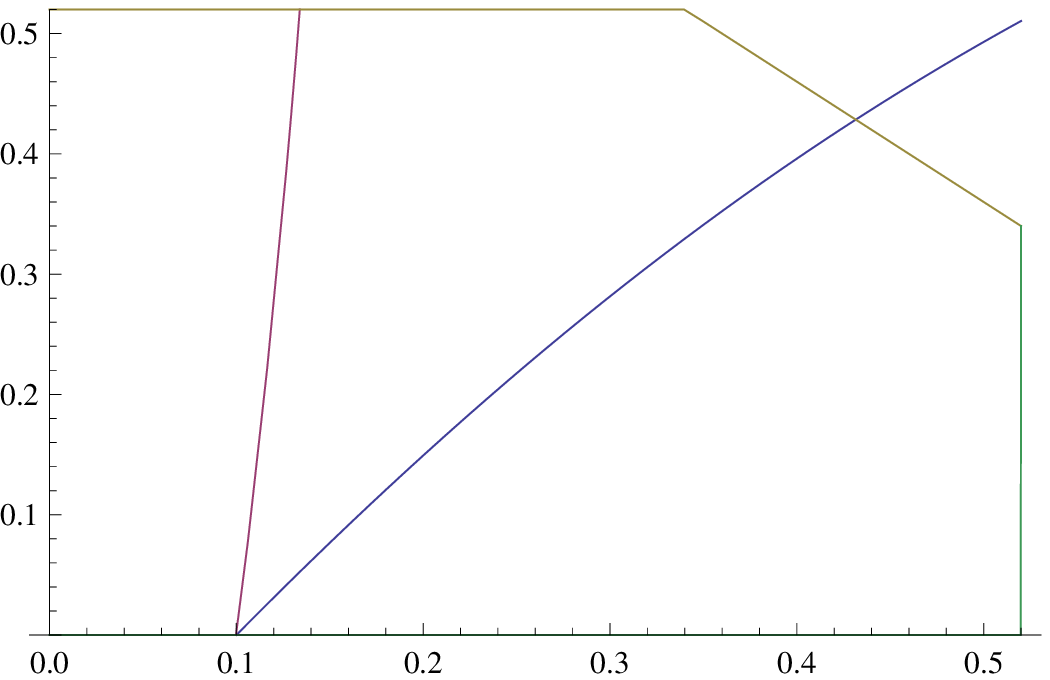}}
\caption{Plots of $f_1$ and $f_2$ for $i=3$ and $i=19$}
\end{figure}

We have $f_1\ge f_2$ and
\[
 (\tfrac{d}{de_0},\tfrac{d}{de_S})\ex(X_i)=\begin{cases}
                                  (<0,>0),&  \mbox{ for }e_0> f_1(e_S),\\
                                  (>0,>0),&  \mbox{ for }f_2(e_S) < e_0 <f_1(e_S),\\
                                  (>0,<0),&   \mbox{ for }e_0< f_2(e_S).
                                   \end{cases}
\]
%
We conclude that the minimum of $\ex(X_i)$ in the considered area with $e_0\le e_S$ occurs in
$
 \left( e_0,e_S\right)= \left( 0,0.52\right)
$
or on the line $e_0=e_S$, and similarly, the minimum of $\ex(X_i)$ in the considered area with $e_0\ge e_S$ occurs in
$
 \left( e_0,e_S\right)= \left( 0.52,0\right)
$
or on the line $e_0=e_S$.
On this line, i.e., $0\le e_0=e_S\le 0.43$, the minimum is attained at 
\[
 e_0=e_S=
\begin{cases}
 0.43, & \mbox{ for }0\le i\le 6,\\
0, & 
\mbox{ for }7\le i\le 19.
\end{cases}
\]

Similarly, the maximum on the line $e_0=0.86-e_S$ is an upper bound for the maximum in the considered area, 
and this maximum is attained at 
\[
\left( e_0,e_S\right)=\left(\tfrac{361+19i}{1350},\tfrac{800-19i}{1350}\right).
\]
Computing the four values for each $i$ with $0\le i\le 19$, we find that (see Appendix~\ref{vals}) we can choose $\delta_i$ and $\hat{\delta}_i$ as follows:
\[
 \begin{array}[]{l|c|c|c|c|c|c|c|c|c|c}
  \hfill i\hfill& 0&1&2&3&4&5&6&7&8&9\\
\hline
\rule{0ex}{2.5ex} 1000~\delta_i&29&26&24&22&21&20&19&18&17&17\\
1000~\delta_{i+10}&18&18&19&20&21&22&24&26&29&31\\
\hline
\rule{0ex}{2.5ex} 1000~\hat{\delta}_i&72&71&70&66&61&56&51&46&42&38\\
1000~\hat{\delta}_{i+10}&34&31&28&25&23&20&18&17&15&14
 \end{array}
\]
satisfying~\eqref{bound1} and~\eqref{bound2}. For this, note that (see Appendix~\ref{vals})
\[
\ex(X_i)\le 
\tfrac{i}{20}s+s-0.22m'<\tfrac{i}{20}s+s-|E(B)|-0.2m'.
\]

Now we are ready to use Azuma's inequality (cf.~\cite{MR}):
\begin{theorem}{(Azuma's inequality)}
 Let $X$ be a random variable determined by $n$ trials $T_1,\ldots ,T_n$, such that for each $j$, and any two
possible sequences of outcomes $t_1,\ldots,t_j$ and $t_1,\ldots,t_{j-1},t_j'$:
\[
 |\ex(X|T_1=t_1,\ldots,T_j=t_j)-\ex(X|T_1=t_1,\ldots,T_j=t_j')|\le c_j,
\]
 then for all $\bar{t},\underline{t}>0$
\[
 \pr(X-\ex (X)\ge \bar{t})+\pr(\ex (X)-X\ge \underline{t})\le e^{-{\bar{t}^2}/(2\sum{c_j^2})}+e^{-{\underline{t}^2}/(2\sum{c_j^2})}. 
\]
\end{theorem}
In our application, $T_j$ is the weight of the $j^{th}$ vertex in $V'$, and 
$X=X_i$. 
As the weight of one vertex in $v\in V'$ changes the value (and thus the expectation) 
of an $X_i$ by at most $d(v)\le \eps m$,
we have 
\begin{multline*}
 \pr(X_i-\ex (X_i)\ge 0.2m')+\pr(\ex (X_i)-X_i\ge tm')\\
\le e^{-(0.2m')^2/({2\sum_{v\in V'}{d(v)^2}})}+e^{-(tm')^2/({2\sum_{v\in V'}{d(v)^2}})}\\
\le  e^{-{(0.2m')^2}/{(2\eps m\sum_{v\in V'}{d(v)})}}+e^{-{(tm')^2}/{(2\eps m\sum_{v\in V'}{d(v)})}}\\
\le  e^{-{(0.04m')(0.99m)}/{4\eps mm'}}+e^{-{(t^2m')(0.99m)}/{4\eps mm'}}\\
= e^{-{0.0099}/{\eps }}+e^{-{0.99t^2}/{4\eps}}.
\end{multline*}
Thus, as $\eps=2.7\times 10^{-5}$,
\begin{multline*}
 \pr(\mbox{inequality \eqref{bound1} or \eqref{bound2} fails for some }i)\\
\le 40e^{-{0.0099}/{\eps }}+\sum_{i=0}^{19}\left( e^{-{0.99\delta_i^2}/{4\eps}}+e^{-{0.99\hat{\delta}_i^2}/{4\eps}}\right) <1.
\end{multline*}

Therefore, there is a choice of the $T_j$ such that none of the 
$X_i$ falls out of the given range. This yields a well guarded vertex weighting.
\qed

\section{Graphs with small maximum degree}\label{maxdegproof}
With the same methods as above we can improve on the bound in Theorem~\ref{brandt} as stated in Theorem~\ref{maxdeg}. Here we give only a proof sketch.
\begin{proof}
Let $\eps =2.3\times 10^{-4}>\frac{1}{4350}$.
 We proceed as in Case~4 of the proof of Theorem~\ref{main} and note that we have $B=\emptyset$, $e_0=e_S=0$ and $m=m'$.
This yields with the same calculations as above, that we can choose
$\delta_i=\hat{\delta}_i$ as follows:
\[
 \begin{array}[]{l|c|c|c|c|c|c|c|c|c|c}
  \hfill i\hfill& 0&1&2&3&4&5&6&7&8&9\\
\hline
1000~\delta_i&94&89&84&80&76&72&69&66&63&60\\
1000~\delta_{i+10}&58&56&55&53&52&52&51&51&52&52
 \end{array}
\]
satisfying~\eqref{bound1} and~\eqref{bound2}.
The same calculation as above involving Azuma's inequality
yields the theorem.
\end{proof}

\bibliographystyle{amsplain}

\begin{thebibliography}{99}

\bibitem{BJMR} M. Ba\u{c}a, S.~Jendrol', M.~Miller and J.~Ryan, On irregular total labelings, {\em Discrete Math.} {\bf 307} (2007), 1378--1388.

\bibitem{BMR} S.~Brandt, J.~Mi\u{s}kuf and D.~Rautenbach, On a Conjecture about Edge Irregular Total Labellings, {\em J. Graph Theory} {\bf 58} (2008), 333--343.

\bibitem{BMR2} S.~Brandt, J.~Mi\u{s}kuf and D.~Rautenbach, Edge Irregular Total Labellings for Graphs of Linear Size, {\em Discrete Math.} {\bf 309} (2009), 3786--3792. 

\bibitem{BRS} S.~Brandt, D.~Rautenbach and M.~Stiebitz, Edge colouring by total labellings, {\em Discrete Math.} {\bf 310} (2010), 199--205.

\bibitem{D} R.~Diestel, Graph Theory, Springer 1997.

\bibitem{IJ} J.~Ivan\u{c}o and S.~Jendrol', Total edge irregulrity strength of trees, {\em Discuss. Math. Graph Theory} {\bf 26} (2006), 449--456.

\bibitem{JMS} S.~Jendrol', J. ~Mi\u{s}kuf and R.~Sot\'{a}k, Total Edge Irregularity Strength of Complete Graphs and Complete Bipartite Graphs, {\em Electronic Notes in Discrete Math.}
{\bf 28} (2007), 281--285 .

\bibitem{MR} M.~Molloy and B.~Reed, Graph Colouring and the Probabilistic Method, Springer 2002.




%


\end{thebibliography}

\appendix

\section{Expected numbers of edges with certain weights}\label{vals}
Let $e^*=\frac{-361+18i+i^2}{20(152-36i+i^2)}$ and $\bar{e}=\frac{361+19i}{1350}$. For given $(e_0,e_S)$, let
$
 Y_i:=\tfrac{1}{m'}~X_i
$.
Then we calculate the following values for $Y_i-\frac{i+1}{60\times 0.99}$ (and in the last column for $\frac{i+20}{60}-Y_i$):
\[
 \begin{array}{|c|c|c|c|c|c|c|}
  \hline 
\rule{0ex}{2.5ex}&\multicolumn{6}{c |}{(e_0,e_S)}\\ \cline{2-7}
\rule{0ex}{2.5ex}\hfill i\hfill& (0,0.52) & (0,0) & (e^*,e^*) &  (0.43,0.43) & (0.52,0) & (\bar{e},0.86-\bar{e}) \\
\hline
 \rule{0ex}{2.5ex}0 & 0.0842642 & 0.0942761 & 0.0945847 &
{\bf0.072}5870 & {\bf0.029}1077 & 0.221914\\
 1 & 0.0780712 & 0.0891370 & 0.0894879 & {\bf0.071}2887 & {\bf0.026}7374 & 0.226687\\
 2 & 0.0721583 & 0.0843057 & 0.0847193 & {\bf0.070}1341 & {\bf0.024}6472 & 0.230987\\
 3 & {\bf0.066}5254 & 0.0797821 & 0.0803057 & 0.0691234 & {\bf0.022}8371 & 0.234814\\
 4 & {\bf0.061}1725 & 0.0755664 & 0.0763530 & 0.0682565 & {\bf0.021}3070 & 0.238167\\
 5 & {\bf0.056}0997 & 0.0716584 & 0.0741222 & 0.0675334 & {\bf0.020}0569 & 0.241046\\
 6 & {\bf0.051}3070 & 0.0680582 & 0.0669080 & 0.0669542 & {\bf0.019}0869 & 0.243452\\
 7 & {\bf0.046}7943 & 0.0647658 & 0.0644259 & 0.0665187 & {\bf0.018}3969 & 0.245385\\
 8 & {\bf0.042}5617 & 0.0617812 & 0.0616330 & 0.0662271 & {\bf0.017}9871 & 0.246844\\
 9 & {\bf0.038}6092 & 0.0591044 & 0.0590379 & 0.0660793 & {\bf0.017}8572 & 0.247830\\
10 & {\bf0.034}9366 & 0.0567354 & 0.0567098 & 0.0660753 & {\bf0.018}0074 & 0.248342\\
11 & {\bf0.031}5442 & 0.0546742 & 0.0546683 & 0.0662151 & {\bf0.018}4377 & 0.248381\\
12 & {\bf0.028}4318 & 0.0529207 & 0.0529207 & 0.0664988 & {\bf0.019}1480 & 0.247946\\
13 & {\bf0.025}5994 & 0.0514750 & 0.0514703 & 0.0669263 & {\bf0.020}1384 & 0.247038\\
14 & {\bf0.023}0471 & 0.0503372 & 0.0503181 & 0.0674976 & {\bf0.021}4088 & 0.245657\\
15 & {\bf0.020}7749 & 0.0495071 & 0.0494643 & 0.0682127 & {\bf0.022}9593 & 0.243802\\
16 & {\bf0.018}7827 & 0.0489848 & 0.0489084 & 0.0690716 & {\bf0.024}7898 & 0.241473\\
17 & {\bf0.017}0705 & 0.0487703 & 0.0486493 & 0.0700744 & {\bf0.026}9004 & 0.238671\\
18 & {\bf0.015}6384 & 0.0488635 & 0.0486852 & 0.0712209 & {\bf0.029}2910 & 0.235396\\
19 & {\bf0.014}4864 & 0.0492646 & 0.0490139 & 0.0725113 & {\bf0.031}9617 & 0.231647\\
\hline
 \end{array}
\]

\section{Polynomials used in Case 4 of the proof}\label{polys}
\begin{align*}
 p_0(i)&= 1444+39i+i^2  \\
p_1(i)&= 10~(-1558-45i+i^2)  \\
p_2(i)&= 600~(19+i)  \\
p_3(i)&=  10~(2660-147i+i^2) \\
p_4(i)&=  600~(-35+i)\\
p_5(i)&=  10~(1216+27i-i^2) \\
p_6(i)&=  600~(19+i) \\
p_7(i)&=  100~(2085136+111336 i- 2663 i^2-78 i^3+ i^4)\\
p_8(i)&=  12000~( -27436-2077i+88 i^2+ i^3)\\
p_9(i)&=  360000~(361+38 i+ i^2) \\
p_{10}(i)&=   1200~(35 -  i)
\end{align*}


\end{document}